\documentclass[a4paper,11pt, twoside]{amsart}

\usepackage{amsmath,amsfonts,amsthm,amsopn,color,amssymb,enumitem}

\usepackage[colorlinks=true,urlcolor=blue, citecolor=red,linkcolor=blue,linktocpage,pdfpagelabels, bookmarksnumbered,bookmarksopen]{hyperref}

\usepackage[left=2.61cm,right=2.61cm,top=2.72cm,bottom=2.72cm]{geometry}

\newcommand{\R}{\mathbb{R}}

\newcommand{\N}{\mathbb{N}}
\newcommand{\RN}{{\mathbb{R}^N}}

\newcommand{\CH}{\mathcal{H}^{p,q}}

\newcommand{\CHZ}{\mathcal{H}_{0,{\rm rad}}^{p,q}}
\newcommand{\CS}{\mathcal{S}_{\rm rad}}
\newcommand{\CSZ}{\mathcal{S}_{0,{\rm rad}}}

\newcommand{\G}{\Gamma}
\renewcommand{\l}{\lambda}

\newcommand{\Div}{\,\mathrm{div}}

\newcommand{\ef}{\eqref}

\makeatletter
\@addtoreset{equation}{section}%
\renewcommand{\theequation}{\thesection.\@arabic\c@equation}
\makeatother

\makeatletter
\providecommand\@dotsep{5}
\def\listtodoname{List of Todos}
\def\listoftodos{\@starttoc{tdo}\listtodoname}
\makeatother

\newtheorem{theorem}{Theorem}[section]
\newtheorem{lemma}[theorem]{Lemma}
\newtheorem{remark}[theorem]{Remark}
\newtheorem{proposition}[theorem]{Proposition}

\newtheorem{corollary}[theorem]{Corollary}

\newcommand{\ad}[1]{{\color{black}#1}}

\title
[
Some quasilinear elliptic equations involving multiple $p$-Laplacians
]
{
Some quasilinear elliptic equations involving multiple $p$-Laplacians
}

\author[A. Pomponio]{Alessio Pomponio}
\author[T. Watanabe]{Tatsuya Watanabe}

\address[A. Pomponio]{\newline\indent
Dipartimento di Meccanica, Matematica e Management
\newline\indent 
Politecnico di Bari
\newline\indent
Via Orabona 4,  70125  Bari, Italy}
\email{alessio.pomponio@poliba.it}

\address[T. Watanabe]{\newline\indent 
Department of Mathematics, 
\newline\indent 
Faculty of Science, Kyoto Sangyo University,
\newline\indent
Motoyama, Kamigamo, Kita-ku, Kyoto-City, 603-8555, Japan}
\email{tatsuw@cc.kyoto-su.ac.jp}

\thanks{}

\date{}

\begin{document}

\begin{abstract}
This paper is devoted to the study, with variational technique, 
of the following quasilinear elliptic problem:
\begin{equation*}
\begin{cases} 
-\Delta_p u - \beta \Delta_q u =g(u) & \hbox{in} \ \RN,
\\
u(x)\to 0 & \hbox{as }|x|\to +\infty,
\end{cases}
\end{equation*}
where $N \ge 3$, $1<p<q$ and $p<N$.
We are interested in the existence of positive solutions for general nonlinearities.
Especially we obtain the existence result for the zero mass case,
which includes a large class of pure power nonlinearities.
More general quasilinear problems of Born-Infeld type are also considered.
\end{abstract}

\maketitle

\section{Introduction}

In this paper, we study, with variational technique, the following quasilinear elliptic problem:
\begin{equation}\label{eq:1.1}\tag{$\mathcal{P}$}
\begin{cases} 
-\Delta_p u - \beta \Delta_q u =g(u) & \hbox{in} \ \RN,
\\
u(x)\to 0 & \hbox{as }|x|\to +\infty,
\end{cases}
\end{equation}
where $N \ge 3$, $\beta >0$, $1<p<q$, $p<N$
and $\Delta_p u=\Div (|\nabla u|^{p-2} \nabla u)$ is the $p$-Laplacian.
In the last decades, a lot of works has been done for the study of $(p,q)$-Laplace equation.
However, most of them are devoted to the bounded domain case 
or problems with critical nonlinearities 
(see, for example, \cite{BCS,CEM,F,MP} and references therein).

If $\beta=0$, \eqref{eq:1.1} reduces to the following scalar field equation:
\begin{equation} \label{eq:1.2}
-\Delta_p u =g(u) \ \hbox{in} \ \RN.
\end{equation}
The existence of solutions of \eqref{eq:1.2} has been studied, among others, in \cite{BL,DM,FG}.
Moreover when $p=2$, the almost optimal condition for the existence of
nontrivial solutions has been obtained in \cite{BL}.
However, a scaling property which plays an essential role in \cite{BL,DM} is lost 
if $\beta \ne 0$ in \ef{eq:1.1},
causing that the approach in \cite{BL,DM} cannot be applied to \ef{eq:1.1}.
Thus it is an challenging problem to look for an optimal condition for the existence
of nontrivial solutions of \ef{eq:1.1}.

The aim of this paper is, therefore, 
to consider \ef{eq:1.1} in the whole $\RN$ and for a general nonlinearity $g$.
Especially, we do not assume any monotonicity conditions on $g$.

Our another motivation comes from the study of the
{\it Born-Infeld} equation which appears in electromagnetism:
\begin{equation} \label{eq:1.3}
-\Div \left( \frac{\nabla u}{\sqrt{1-\frac{1}{b^2}|\nabla u|^2}} \right) =g(u) 
\ \mbox{in} \ \RN,
\end{equation}
where $b$ is a positive constant and called {\it the absolute field}  constant. 
(We refer to \cite{BDP} and references therein 
for more physical backgrounds of the Born-Infeld equation.)
Indeed by the Taylor expansion, it follows that
$$
\frac{1}{\sqrt{1-x}}=1+\frac{x}{2}+\frac{3}{2\cdot 2^2}x^2+\frac{5!!}{3!\cdot 2^3}x^3
+\cdots +\frac{(2k-3)!!}{(k-1)!2^{k-1}}x^{k-1}+\cdots \quad 
\hbox{for} \ |x|<1.$$
Putting $x=\frac{|\nabla u|^2}{b^2}$ and $\beta=\frac{1}{2b^2}$ formally,
we can see that the $0$-th order approximated problem of \ef{eq:1.3} is exactly 
the scalar field equation \ef{eq:1.2} with $p=2$.
When we adopt the $1$-st order approximation, 
one has the following quasilinear elliptic equation:
\begin{equation*} 
-\Delta u- \beta \Delta_4 u = g(u) \quad \hbox{in} \ \RN,
\end{equation*}
which can be obtained by taking $p=2$ and $q=4$ in \ef{eq:1.1}.
Furthermore the $k$-th order approximated problem is given by
\begin{equation} \label{eq:1.4}
-\Delta u-\beta \Delta_4u -\frac{3}{2} \beta^2\Delta_6 u 
-\cdots -\frac{(2k-3)!!}{(k-1)!}\beta^{k-1}\Delta_{2k} u =g(u) \quad \hbox{in} \ \RN,
\end{equation}
where $k \in \N$, $(2k-3)!!=(2k-3)(2k-5)\cdots 5 \cdot 3\cdot 1$, 
$(-1)!!=1$. Thus it is natural to ask if solutions of \ef{eq:1.3}
can be obtained as a limit of solutions for \ef{eq:1.4}.
This question has been considered in \cite{BDP} for the inhomogeneous Born-Infeld problem:
\begin{equation} \label{eq:1.5}
-\Div \left( \frac{\nabla u}{\sqrt{1-\frac{1}{b^2}|\nabla u|^2}} \right) = \rho(x) 
\quad \mbox{in} \ \R^N.
\end{equation}
It is shown that under suitable assumptions on $\rho$, 
the unique minimizer of the action functional associated to \ef{eq:1.5} 
can be obtained as a weak limit of
the unique solution of the $k$-th order approximated problem for \ef{eq:1.5}.
(See \cite[Theorem 5.2]{BDP} and \cite{CK} for related results.)
On the other hand, problem \ef{eq:1.3} is much less studied. 
In \cite{BDD}, the case $g(u)=|u|^{\alpha-2}u$ with
$\alpha>\frac{2N}{N-2}$ has been considered.
Then it was shown that \ef{eq:1.3} has a positive radial solution
and a sequence of radial solutions. 
Moreover, again by restricting the research to solutions with radial symmetry, 
in \cite{A2} the equation \eqref{eq:1.3} is reduced to an ODE for which the existence, 
non-existence and multiplicity of ground states 
(namely positive solutions going to zero at infinity) 
and bound states (i.e. solutions going to zero at infinity) 
are investigated for the Lane-Emden type equation. 
By the use of the shooting method, 
in \cite{A1} the existence of a ground state solution is also determined 
for the equation presenting a sign-changing nonlinearity.

Our purpose of this paper is to investigate the existence of positive solutions 
of \ef{eq:1.1} and \ef{eq:1.4} for a wide class of nonlinearities including the case $g(u)=u^{\alpha}$.
We expect our existence results will be the next step for the further study 
of the Born-Infeld equation \ef{eq:1.3}.
Hereafter in this paper, we take $\beta=1$ for simplicity, 
since $\beta$ plays no essential role in the study of the existence of solutions.

In order to consider general nonlinear terms, we have to take into account 
behavior of $g(s)$ near zero and infinity.
For the problem \ef{eq:1.2} with $p=2$ and in  the {\it positive mass} case,
namely when $g(s)$ satisfies
$$-\infty<\liminf_{s\to 0^+} \frac{g(s)}{s} \le
\limsup_{s \to 0^+} \frac{g(s)}{s} =-m \ \mbox{for some} \ m>0,$$
almost optimal condition for the existence of nontrivial solutions has been obtained in \cite{BL}. 
(See also \cite{HIT}.) Conversely, whenever $m=0$,
the so called {\it zero mass} case, some results are contained, 
among others, in \cite{Ta}, if $g$ corresponds to the critical power
$s^{(N+2)/(N-2)}$, and in \cite{AP07,BPR,BL}, if $g$ is
supercritical near the origin and subcritical at infinity (see
also \cite{BM} for the case of exterior domain and \cite{AP3} for
complex valued solutions).

We anticipate that our problem has two quite interesting features: 
we can treat the zero mass case and the positive mass one in a similar way and, 
moreover, in the zero mass case, we can treat several pure power nonlinearities. 
This is due to the particular functional setting that we will introduce to study \eqref{eq:1.1}. 
Indeed, while in presence of a single $p$-laplacian the natural framework is $D^{1,p}(\RN)$, 
namely the completion of $C^\infty_0(\RN)$ with the respect of the $L^p$-norm of the gradient, 
and we know that $D^{1,p}(\RN)$ is embedded only into $L^{p^*}(\RN)$, 
where $p^*=(pN)/(N-p)$, in our case we will introduce a combination of Sobolev spaces, 
a sort of intersection between $D^{1,p}(\RN)$ and $D^{1,q}(\RN)$, 
which guarantees suitable embeddings properties into a large range of Lebesgue spaces 
(see Section \ref{se:vs} for more details). 
Finally, we would like to stress that the unique assumption on $q$ is that 
it is strictly greater than $p$ but it can be large as we want, 
as this is a great help in order to better approximate the Born-Infeld operator.

\medskip

We can now introduce our precise assumptions and results 
and we start dealing with the zero mass case.
The following hypotheses can be regarded as a natural extension 
of the zero mass case for \ef{eq:1.2} to the quasilinear problem \ef{eq:1.1}. 

On the nonlinearity $g$, we require that
\begin{enumerate} [label=(g\arabic*),ref=g\arabic*]
\item \label{g1} $g\in C(\R,\R)$, $g(s) \equiv 0$ if $s \le 0$;
\item \label{g2} for all $\ell\in [p,p^*]$, it holds
$$ 
-\infty 
\le \limsup_{s \to 0^+} \frac{g(s)}{s^{\ell-1}} \le 0,$$
where $p^*=\frac{pN}{N-p}\in (p,+\infty)$;
\item \label{g3}
if $q<N$, it holds that
\begin{equation} \label{eq:1.6}
\displaystyle -\infty \le \limsup_{s \to +\infty}
\frac{g(s)}{s^{q^*-1}} \le 0,
\end{equation}
where $q^*=\frac{qN}{N-q}\in (p^*,+\infty)$; instead, if $q \ge N$, 
we assume \ef{eq:1.6} holds for some $q^*> \max{ \{q,p^*\}}$;
\item \label{g4} there exists $\zeta>0$ such that $G(\zeta)=\int_0^{\zeta} g(s) \,ds >0$.
\end{enumerate}

As we will see in Section \ref{se:vs}, 
the exponents $p^*$ and $q^*$ appear naturally if we consider embedding theorems for
energy spaces associated with \ef{eq:1.1}.
Especially, if $q<N$, since $q^*$ can be seen as a critical exponent for \ef{eq:1.1},
the condition (\ref{g3}) implies that the nonlinear term $g(s)$ has $W^{1,q}$-subcritical
growth at infinity.
On the other hand, (\ref{g2}) means that $g(s)$ has zero mass,
as well as $W^{1,p}$-supercritical growth near zero.
Here we list typical examples of $g(s)$.
\begin{itemize}
\item $g(s)= \min \{ |s|^{q^*-2}s, |s|^{\ell-2}s \}$ for $p^* < \ell < q^*$.
\item $g(s)= |s|^{\ell-2}s$ for $p^* < \ell < q^*$.
\item $g(s)= K|s|^{\ell_1-2}s-|s|^{\ell_2-2}s$ for $p^*<\ell_1 \le q^*$, 
$\ell_1<\ell_2$ and large $K>0$.
\item $g(s)= -|s|^{\ell_1-2}s+|s|^{\ell_2-2}s$ for $p^* \le \ell_1<\ell_2<q^*$.
\end{itemize}
As the second example shows, we can consider a large class of pure power nonlinearities 
for our problem \ef{eq:1.1}, which is impossible for \ef{eq:1.3}.
Especially in the case $q>N$, let $\ell>\frac{pN}{N-p}$ be arbitrarily given
and consider the nonlinear term $g(s)=s^{\ell-1}$.
Then choosing any $q^*> \max \{ \ell,q \}$, 
we see that \eqref{g1}-\eqref{g4} are all satisfied.
In this setting, we have the following result.
\begin{theorem}\label{thm:1.1}
Assume \eqref{g1}-\ad{\eqref{g4}}. Then problem \ef{eq:1.1} has a solution
which is positive and radially symmetric and belongs to $C_{\rm loc}^{1,\sigma}(\RN)\cap L^\infty (\RN)$, 
for some $\sigma \in (0,1)$.
\end{theorem}
Moreover we will prove that there exists a {\em radial ground state solution}, 
namely a solution of \eqref{eq:1.1} which minimizes the action functional 
among all nontrivial radial solutions of \eqref{eq:1.1} 
(see Theorem \ref{thm:2.7} for the precise statement).

Next we state a result for the positive mass case for \ef{eq:1.1}.
In this case, we assume the following condition \ad{instead of assumption \eqref{g2}}:
\begin{enumerate} 
[label=(g\arabic*'),ref=g\arabic*']
\setcounter{enumi}{1}
\item \label{g2'}there exist $\ell\in [p,p^*)$ and $m_{\ell}>0$ such that 
\begin{equation} \label{eq:1.7}
-\infty<\liminf_{s\to 0^+} \frac{g(s)}{s^{\ell-1}} \le
\limsup_{s \to 0^+} \frac{g(s)}{s^{\ell-1}} =-m_\ell.
\end{equation}
\end{enumerate}
We note that if \eqref{eq:1.7} holds for $\ell=p^*$, (g2) is fulfilled. 
Then we obtain the following result.

\begin{theorem}\label{thm:1.2}
Assume \eqref{g1}, \eqref{g2'}, \eqref{g3} and \eqref{g4}. 
Then problem \ef{eq:1.1} has a solution
which is positive and radially symmetric and belongs to $C_{\rm loc}^{1,\sigma}(\RN)\cap L^\infty (\RN)$, 
for some $\sigma \in (0,1)$.
\end{theorem}

Also in this case, the existence of a radial ground state solution of \ef{eq:1.1} 
can be also obtained, see Theorem \ref{thm:3.5} below.

We believe that assumptions \eqref{g1}, \eqref{g2}-\eqref{g2'}, \eqref{g3} and \eqref{g4}
are almost optimal for the existence of non-trivial solutions of 
\ef{eq:1.1} when $q<N$.
We finally remark that very general quasilinear elliptic equations have been treat also, 
for example, in \cite{ADP,FLS,SS}, but our problem does not fall in the studied cases.

This paper is organized as follows. In Section \ref{se:0}, we consider the zero mass case.
First we prepare embedding theorems for energy spaces associated with \ef{eq:1.1}
and perform the variational setting in Section \ref{se:vs}.
Next in Section \ref{se:exi}, we prove the existence of a positive radial solution of \ef{eq:1.1}
by using the Mountain Pass Theorem together with Jeanjean's Monotonicity trick \cite{J}.
We consider the positive mass case in Section \ref{se:pos} and, 
finally, we devote the Section \ref{se:appr} to the study 
of the $k$-th order approximated problem \eqref{eq:1.4}. 

\section{The zero mass case} \label{se:0}

\subsection{Variational setting and preliminaries} \label{se:vs}

\ 

\smallskip
In this section, we give some preliminaries. 
First, we introduce the framework where we will study \ef{eq:1.1} 
and present its  embedding properties.
Next, we introduce the energy functional associated with \ef{eq:1.1} and
modify the nonlinear term in order to find a nontrivial critical point.

We work on the functional space $\CH_0$ which is
given by $\CH_0= \overline{ C_0^{\infty}(\RN)}^{\ \| \, \cdot \, \|_{\CH_0}}$, where
\[
\| u\|_{\CH_0}:= \|\nabla u\|_p +\|\nabla u\|_q.
\]

In the following theorem we study the embeddings properties of $\CH_0$.
\begin{theorem}\label{th:embedding}
Let $1<p<q$ and $p<N$. Then 
\[
\CH_0 \hookrightarrow L^r(\RN), \quad \hbox{ for any } \
\frac{pN}{N-p}\le r \ 
\begin{cases}
\le \frac{qN}{N-q} & \hbox{if }q<N,
\\
< + \infty & \hbox{if } q=N, \\
\le +\infty & \hbox{if } q> N.
\end{cases}
\]
\end{theorem}

\begin{proof}
We distinguish three different cases.

\medskip
\noindent{\sc Case  $1<p<q<N$.}
\\
By standard Sobolev inequalities, we have that
$$
\| u\|_{\frac{pN}{N-p}} \le C \| \nabla u\|_p \le C \|u\|_{\CH_0},
\quad
\| u\|_{\frac{qN}{N-q}} \le C \| \nabla u\|_q\le C \|u\|_{\CH_0}
$$
and so
\begin{equation*}
\CH_0 \hookrightarrow L^{\frac{pN}{N-p}}(\RN)\cap L^{\frac{qN}{N-q}}(\RN).
\end{equation*}

\medskip
\noindent
{\sc Case  $1<p<q=N$.}
\\
Going back the proof of the Sobolev inequality, if $u\in C_0^{\infty}(\RN)$, one has
\begin{equation} \label{eq:a.1}
\| u\|_{\frac{N}{N-1}} \le \prod_{i=1}^N 
\left\| \frac{\partial u}{\partial x_i} \right\|_1^{\frac{1}{N}}.
\end{equation}
(See \cite[(19), P.\,280]{B}.)
Let $m \ge 1$. Applying \ef{eq:a.1} to $|u|^{m-1}u$, we get
$$
\| u\|_{\frac{mN}{N-1}}^m
\le m \prod_{i=1}^N \left\| |u|^{m-1} \frac{\partial u}{\partial x_i} \right\|_1^{\frac{1}{N}}
\le C \| \nabla u\|_N \| u\|_{\frac{(m-1)N}{N-1}}^{m-1}.$$
By the Young inequality, it follows that
\begin{equation} \label{eq:a.2}
\| u\|_{\frac{mN}{N-1}} \le C( \| u\|_{\frac{(m-1)N}{N-1}} + \| \nabla u\|_N)
\quad \hbox{for any} \ m \ge 1.
\end{equation}
In \ef{eq:a.2}, we first choose $\frac{(m-1)N}{N-1}=\frac{pN}{N-p}$, that is, 
$m= \frac{(N-1)p}{N-p}+1$. Writing $p^*=\frac{pN}{N-p}$ for simplicity, 
one has $m=\frac{N-1}{N}p^*+1$ and $\frac{mN}{N-1}=p^*+\frac{N}{N-1}$.
Thus from \ef{eq:a.2}, we obtain
$$
\| u\|_{p^*+\frac{N}{N-1}}
\le C( \| u\|_{p^*} + \| \nabla u\|_N)
\le C( \| \nabla u\|_p + \| \nabla u\|_N).$$
Iterating this procedure with $m=\frac{N-1}{N}p^*+j$ for $j \in \mathbb{N}$,
and applying the interpolation inequality, one gets
$$
\| u\|_r \le C( \| \nabla u\|_p +\| \nabla u\|_N)
\quad \hbox{for all} \ u\in C_0^{\infty}(\RN) 
\ \hbox{and} \ r\in[p^*,+\infty).$$
This completes the proof by a density argument.

\medskip
\noindent
{\sc Case  $1<p<N<q$.}
\\
We argue as in \cite{FOP}. Let  $u\in C_0^{\infty}(\RN)$, 
$x\in \RN$ and $Q$ be an open cube, containing $x$, 
whose sides -of length $1$- are parallel to the coordinate axes.
Going back to the proof of the Morrey inequality, we have
$$
|\bar{u}-u(x)| \le \frac{q}{q-N} \| \nabla u\|_{L^q(Q)},$$
where $\bar{u}=\frac{1}{|Q|}\int_Q u(x) \,dx$.
(See \cite[(27), P.\,283]{B} for the proof.)
By the H\"older inequality, we arrive at
\begin{align*}
|u(x)| &\le \left| \frac{1}{|Q|} \int_Q u(x)\,dx \right| +C \| \nabla u\|_{L^q(Q)} 
\le C \| u\|_{L^{p^*}(Q)} + C \| \nabla u\|_{L^q(Q)} \\
&\le C( \| u\|_{L^{p^*}(\RN)} + \| \nabla u\|_{L^q(\RN)})
\le C( \| \nabla u\|_{L^p(\RN)} + \| \nabla u\|_{L^q(\RN)}),
\end{align*}
from which we deduce that
\begin{equation*}
\| u\|_{\infty} \le C (\| \nabla u\|_p+\| \nabla u\|_q).
\end{equation*} 
Again, we conclude by a density argument.
\end{proof}

\begin{remark}
By Theorem \ref{th:embedding}, for any $1<p<q$ with $p<N$, 
according with the definitions of $p^*$ and $q^*$ given in the Introduction, one has
\begin{equation}\label{immersione}
\CH_0 \hookrightarrow L^r(\RN)
\ \hbox{for any} \ r\in [p^*,q^*].
\end{equation}
\end{remark}
Moreover the following property will be useful later
\begin{proposition}
Let $1<p<q$ with $p<N$. Then, for any  $\ u\in \CH_0$ and $r\in [p^*,q^*]$, 
we have
\begin{equation} \label{eq:2.4}
\| u\|_r^r \le C( \| \nabla u\|_p^r+\| \nabla u\|_p^{p^*}+\| \nabla u\|_q^{q^*}).
\end{equation}
\end{proposition}

\begin{proof}
Let $u\in \CH_0$ and $r\in [p^*,q^*]$. 
By Theorem \ref{th:embedding}, the interpolation inequality and the Young inequality, we get
\begin{align*} 
\| u\|_r^r &\le \| u\|_{p^*}^{\theta p^*} \| u\|_{q^*}^{(1-\theta)q^*}
\le C \| \nabla u\|_p^{\theta p^*} ( \| \nabla u\|_p+\| \nabla u\|_q)^{(1-\theta)q^*} \\
&\le C \| \nabla u\|_p^{\theta p^*}
( \| \nabla u\|_p^{(1-\theta)q^*}+\| \nabla u\|_q^{(1-\theta)q^*}) \\
&\le C( \| \nabla u\|_p^r+\| \nabla u\|_p^{p^*}+\| \nabla u\|_q^{q^*}).
\end{align*}
Here $\theta\in [0,1]$ is a constant chosen so that $r=\theta p^* +(1-\theta)q^*$.
\end{proof}

Let us define the functional $I:\CH_0 \to \R$ by
$$
I(u)=\frac{1}{p} \|\nabla u\|_p^p
+\frac{1}{q} \|\nabla u\|_q^q 
-\int_{\RN} G(u) \,dx.$$
By hypotheses \ad{\eqref{g1}-\eqref{g3}}, we can see that $I$ is well-defined 
and of class $C^1$ on $\CH_0$.
Moreover any critical points of $I$ are solutions of \ef{eq:1.1}.

Next we truncate and decompose the nonlinear term $g$ similarly as in \cite{BL}. 
Let us put
$$
s_0:= \min\{ s\in [\zeta,+\infty) \mid  g(s)=0 \}$$
and $s_0=+\infty$ if $g(s) \ne 0$ for all $s \ge \zeta$. 
We define $\tilde{g}:\R_+ \to \R$ by
$$
\tilde{g}(s)=\left\{
\begin{array}{ll}
g(s) & \hbox{on} \ [0,s_0],\\
0 & \hbox{on} \ (s_0,+\infty).
\end{array}
\right.$$
By the maximum principle, any positive solutions of \ef{eq:1.1} with $\tilde{g}$ 
satisfy the original problem \ef{eq:1.1}. 
Thus we may replace $g$ by $\tilde{g}$ in \ef{eq:1.1}. 
Hereafter we write $g=\tilde{g}$ for simplicity. 
For $s \ge 0$, we set
$$
g_1(s):= g_+(s), \ \hbox{ and } \ 
g_2(s):=g_1(s)-g(s).$$
Then by \eqref{g2} and \eqref{g3}, one has
\begin{equation} \label{eq:2.7}
\lim_{s \to 0} \frac{g_1(s)}{s^{p^*-1}}=0, \quad
\lim_{s \to +\infty} \frac{g_1(s)}{s^{q^*-1}}=0.
\end{equation}
Thus, for $s \ge 0$, we have from \eqref{eq:2.7} that
\begin{align}
0 &\le \ad{ g_1(s)\le C(s^{p^*-1}+s^{q^*-1})}, \label{eq:2.5}
\\
0&\le g_2(s). \label{eq:2.6}
\end{align}
Hence, denoting $G_i(t)=\int_0^t g_i(s) \,ds$ for $i=1,2$, we get
\begin{equation} \label{eq:2.8}
G_2(s) \ge 0 \ \hbox{for all} \ s\in \R,
\end{equation}
and
\begin{equation}\label{eq:2.9}
\ad{ 0\le G_1(s) \le C (|s|^{p^*}+|s|^{q^*})}
\ \hbox{for all} \ s\in \R.
\end{equation}

\medskip
\subsection{Existence of a positive solution of \ef{eq:1.1}}\label{se:exi}

 \ 
\medskip

In all this section, we assume \eqref{g1}-\ad{\eqref{g4}} and prove Theorem \ref{thm:1.1}. 
To this end, we consider the following auxiliary problem:
\begin{equation} \label{eq:2.10}
-\Delta_p u- \Delta_q u+g_2(u) =\lambda g_1(u) \ \hbox{in} \ \RN
\end{equation}
for $\lambda$ close to $1$. 
Our strategy is to find a solution of \ef{eq:2.10} and pass the limit $\lambda \nearrow 1$. 
We define the functional $I_{\lambda}:\CH_0 \to \R$ by
$$
I_{\lambda}(u)=\frac{1}{p} \|\nabla u\|_p^p
+\frac{1}{q} \|\nabla u\|_q^q 
+\int_{\RN} G_2(u) \,dx
-\lambda \int_{\RN} G_1(u) \,dx.$$
In order to find a non-trivial critical point of $I_{\lambda}$, 
we apply a slightly modified version of the Monotonicity trick
due to \cite{J} (see also \cite{AP}). 

\begin{proposition}[Monotonicity trick] \label{prop:2.1}
Let $\big(X,\|\cdot\|\big)$ be a Banach space and $J\subset\R^+$ an interval.
Consider a family of $C^1$ functionals $I_{\lambda}$ on $X$ defined by
\begin{equation*}
I_\l(u)=A(u)- \l B(u) \ \hbox{for} \ \l\in J,
\end{equation*}
with $B$ non-negative and either $A(u)\to + \infty$ or
$B(u)\to+\infty$ as $\|u\|\to+\infty$ and such that $I_\l(0)=0$.
For any $\l\in J$, we set
\begin{equation} \label{eq:2.11}
\Gamma_\l:=\{\gamma\in C([0,1],X)\mid \gamma(0)=0, \ I_\l(\gamma(1))< 0\}. 
\end{equation}
Assume that for every $\l\in J$, the set $\G_\l$ is non-empty and
\begin{equation} \label{eq:2.12}
c_\l:=\inf_{\gamma\in\Gamma_\l}\max_{t\in[0,1]} I_\l(\gamma(t)) >0.
\end{equation}
Then for almost every $\l\in J$, there is a sequence $\{ v_n \} \subset X$ such that
\begin{itemize}
\item[\rm(i)] $\{ v_n \}$ is bounded in $X$;
\item[\rm(ii)] $I_\l(v_n)\to c_\l$;
\item[\rm(iii)] $(I_\l)'(v_n)\to 0$ in the dual space $X^{-1}$ of $X$.
\end{itemize}
\end{proposition}

In our case, we set $X=\CH_{0,{\rm rad}}$, where
\[
\CH_{0,{\rm rad}}=\{u\in \CH_0\mid u \hbox{ is radially symmetric }\},
\] 
and 
\begin{align*}
A(u)&=\frac{1}{p} \|\nabla u\|_p^p
+\frac{1}{q} \|\nabla u\|_q^q 
+\int_{\RN} G_2(u) \,dx,
\\
B(u) &=\int_{\RN} G_1(u) \,dx.
\end{align*}
To apply Proposition \ref{prop:2.1}, we begin with the following lemma.
\begin{lemma}\label{lem:2.2}
There exists $\lambda_0 \in (0,1)$ such that 
the set $\Gamma_\l$ defined in \ef{eq:2.11} is non-empty 
for every $\lambda \in J=[\lambda_0,1]$.
\end{lemma}

\begin{proof}
First by \eqref{g4}, there exists $z \in \CH_{0,{\rm rad}}$ such that
$ \displaystyle \int_{\RN} G(z) \,dx >0$.
(See \cite[Proof of Theorem 2, P. 325]{BL}.)
Since $G(s)=G_1(s)-G_2(s)$, there exists $0<\lambda_0<1$ such that
\begin{equation}\label{eq:2.13}
\lambda_0 \int_{\RN} G_1(z) \,dx
-\int_{\RN} G_2(z) \,dx>0.
\end{equation}
Let $\lambda \in J=[\lambda_0,1]$ and $t>0$. We compute
$I_{\lambda}\left( z( \frac{\cdot}{t} )\right)$.
From \ef{eq:2.13}, one has
\begin{align*}
I_{\lambda}\left( z \left( \frac{\cdot}{t} \right)\right)
&=
\frac{t^{N-p}}{p} \|\nabla z\|_p^p 
+\frac{t^{N-q}}{q} \|\nabla z\|_q^q 
+t^N \int_{\RN} G_2(z)\,dx
-\lambda t^N \int_{\RN} G_1(z) \,dx\\
&\le 
\frac{t^{N-p}}{p} \|\nabla z\|_p^p \,dx
+\frac{t^{N-q}}{q} \|\nabla z\|_q^q \,dx 
-t^N \left(
\lambda_0 \int_{\RN} G_1(z) \,dx
-\int_{\RN} G_2(z) \,dx \right) .
\end{align*}
Hence, we can choose $\tau>1$ so that $I_{\lambda}\left( z( \frac{\cdot}{\tau})
\right)<0$ and consider a function $\gamma:[0,1] \to \CH_{0,{\rm rad}}$ which is defined by 
$$
\gamma(t) = 
\left\{
\begin{array}{ll}
2t z \left( \dfrac{2 \ \cdot}{\tau} \right) & \hbox{if }t\in [0,1/2],
\\
\
\\
z \left( \dfrac{\cdot}{t\tau} \right) & \hbox{if }t\in [1/2,1].
\end{array}
\right.$$
Then it follows that $\gamma \in \Gamma_\l$ and hence the proof is complete.
\end{proof}

\begin{lemma}\label{lem:2.3}
For all $\lambda \in J=[\lambda_0,1]$, the condition \ef{eq:2.12} holds.
\end{lemma}

\begin{proof}
For any $u\in \CH_{0,{\rm rad}}$ and $\lambda \in J$, 
we have from \ef{eq:2.8} and \ef{eq:2.9} that
$$
I_{\lambda}(u) \ge \frac{1}{p} \| \nabla u\|_p^p
+\frac{1}{q}\| \nabla u\|_q^q \ad{ -C \|u \|_{p^*}^{p^*} -C \| u\|_{q^*}^{q^*} }.$$
Thus by \ad{ applying \ef{eq:2.4} with $r=p^*$ and $r=q^*$ respectively}, one gets
\begin{align*}
I_{\lambda}(u) &\ge 
\frac{1}{p} \| \nabla u\|_p^p +\frac{1}{q} \| \nabla u\|_q^q
-C \big( \| \nabla u\|_p^{p^*} + \| \nabla u \|_q^{q^*} \big)
-C( \| \nabla u\|_p^{p^*} + \| \nabla u\|_p^{q^*} + \| \nabla u\|_q^{q^*} \big) \\
&\ge \frac{1}{p} \| \nabla u\|_p^p
\Big(1-C\| \nabla u\|_p^{p^*-p}-C \| \nabla u\|_p^{q^*-p}\Big)
+\frac{1}{q} \| \nabla u\|_q^q \Big( 1-C\| \nabla u\|_q^{q^*-q}\Big).
\end{align*}
Let $u \in \CH_{0,{\rm rad}}$ be such that $\| u\|_{\CH_0}=\| \nabla u\|_p+\| \nabla u\|_q=\rho<1$.
Since \ad{$q^* > p^*>p$} and $q^*>q$, if $\rho>0$ is sufficiently small, 
it follows that
$$
\ad{1-C\| \nabla u\|_p^{p^*-p}-C \| \nabla u\|_p^{q^*-p}\ge \frac{1}{2}}, \quad
1-C\| \nabla u\|_q^{q^*-q} \ge \frac{1}{2}.$$
Then from $p<q$ and $\| \nabla u\|_p \le \| u\|_{\CH_0} <1$, we obtain
$$
I_{\lambda}(u) \ge \frac{1}{2p} \| \nabla u\|_p^p
+\frac{1}{2q}\| \nabla u \|_q^q
\ge \frac{1}{2p} \| \nabla u\|_p^q
+\frac{1}{2q} \| \nabla u\|_q^q
\ge C \| u\|_{\CH_0}^q.$$
Thus there exists $\delta>0$ such that 
$I_\l(u)\ge \delta$ for all $u\in \CH_{0,{\rm rad}}$ with $\|u\|_{\CH_0}\le\rho$.
\\
Now we fix $\lambda \in J$ and $\gamma \in \Gamma_\lambda$. 
Since $\gamma(0)=0 \ne \gamma(1)$ and $I_{\lambda}(\gamma(1)) < 0$,
it follows that $ \| \gamma(1) \|_{\CH_0} > \rho$.
By continuity, we deduce that there exists $t_{\gamma} \in (0,1)$ such that
$\| \gamma(t_{\gamma})\| =\rho$.
Thus for any $\lambda \in J$, we obtain
$$
\alpha \le \inf_{\gamma \in \Gamma} I_{\lambda}(\gamma(t_{\lambda}))
\le c_{\lambda}.$$
This completes the proof.
\end{proof}


By Lemmas \ref{lem:2.2} and \ref{lem:2.3}, we can apply 
Proposition \ref{prop:2.1} to obtain a bounded Palais-Smale
sequence $\{ u_n^{\lambda}\} \subset \CH_{0,{\rm rad}}$ of $I_{\lambda}$ for almost every $\l\in J$, that is,
$$
I_{\lambda}(u_n^{\lambda}) \to c_{\lambda}, \ 
I_{\lambda}'(u_n^{\lambda}) \to 0 \ \hbox{and} \ 
 \{u_n^{\lambda}\} \ \hbox{is bounded in }\CH_0
$$
Hence, passing to a subsequence, there exists $u_{\lambda} \in \CH_{0,{\rm rad}}$ such that
\begin{align}
&u_n^{\lambda} \rightharpoonup u_{\lambda} \ \hbox{in} \ \CH_0,
\ \hbox{ as} \ n \to +\infty \nonumber
\\
\label{eq:2.14}
&u_n^{\lambda}(x) \to u_{\lambda}(x) \ 
\hbox{a.e.} \ x\in \RN, \ \hbox{ as} \ n \to +\infty.
\end{align}

\begin{lemma}\label{lem:2.5}
The weak limit $u_{\lambda}$ satisfies
$$
u_{\lambda} \ne 0, \ I_{\lambda}'(u_{\lambda})=0 \ \hbox{and} \
I_{\lambda}(u_{\lambda}) \le c_{\lambda}.$$
\end{lemma}

\begin{proof}
First we claim that 
\begin{equation}\label{eq:2.15}
\int_{\RN} G_1(u_n^{\lambda}) \,dx
\to \int_{\RN} G_1(u_{\lambda}) \,dx,
\end{equation}
\begin{equation}\label{eq:2.16}
\int_{\RN} g_1(u_n^{\lambda})u_n^{\lambda} \,dx
\to \int_{\RN} g_1(u_{\lambda})u_{\lambda} \,dx,
\end{equation}
and, for any $\varphi \in C_0^{\infty}(\RN)$ and for $i=1,2$,
\begin{equation}\label{eq:2.17}
\int_{\RN} g_i(u_n^{\lambda}) \varphi \,dx
\to \int_{\RN} g_i(u_{\lambda}) \varphi \,dx.
\end{equation}
To this end, we apply the compactness lemma due to Strauss. (See Lemma \ref{lem:a.2} below.)
\\
Let $Q(s)=|s|^{p^*}+|s|^{q^*}$.
Then from \ef{eq:2.7}, it follows that
$\frac{G_1(s)}{Q(s)} \to 0$ as $s \to 0$ and $s \to \infty$.
Moreover from \ef{eq:2.14}, we also have
$G_1( u_n^{\lambda}(x)) \to G_1(u_{\lambda}(x))$ a.e. $x\in \RN$ and, by \eqref{immersione}
$$
\sup_{n \in \N} \int_{\RN} Q(u_n^{\lambda}) \,dx
\le C \sup_{n \in \N} 
\left( \| u_n^{\lambda} \|_{\CH_0}^{p^*}
+\| u_n^{\lambda} \|_{\CH_0}^{q^*} \right)
<+\infty.$$
Finally since $u_n^{\lambda} \in \CH_{0, {\rm rad}}\subset D^{1,p}_{\rm rad}(\RN)$, 
we have, by the radial lemma (see Lemma \ref{lem:a.1} below), that
$u_n^{\lambda}(x) \to 0$ as $|x| \to +\infty$ uniformly in $n\in \N$.
Thus all assumptions in Lemma \ref{lem:a.2} are satisfied.
Then it follows that $G_1(u_n^{\lambda}) \to G_1(u_{\lambda})$ in $L^1(\RN)$
and hence \ef{eq:2.15} holds.
Arguing similarly, one can show that \ef{eq:2.16} and \ef{eq:2.17}.
\\
Now from \ef{eq:2.17}, $I'_{\lambda}(u_n^{\lambda}) \to 0$ and
$u_n^{\lambda} \rightharpoonup u_{\lambda}$ in $\CH_0$,
one has $I_{\lambda}'(u_{\lambda})=0$.
To prove $u_{\lambda} \ne 0$, we suppose by contradiction that
$u_{\lambda}=0$. Since $I_{\lambda}'(u_n^{\lambda}) \to 0$, 
we have by the boundedness of $\{ u_n^{\lambda}\}$ in $\CH_0$ that
$$
\|\nabla u_n^{\lambda}\|_p^p
+ \| \nabla u_n^{\lambda} \|_q^q+\int_{\RN} g_2(u_n^{\lambda})u_n^{\lambda} \,dx
=\lambda \int_{\RN} g_1(u_n^{\lambda}) u_n^{\lambda} \,dx
+o(1).$$
Then from \eqref{eq:2.6} and \ef{eq:2.16}, 
it follows that $\| u_n^{\lambda} \|_{\CH_0} \to 0$, which contradicts 
$I_{\lambda}(u_n^{\lambda}) \to c_{\lambda}>0$.
\\
Finally we show that $I_{\lambda}(u_{\lambda}) \le c_{\lambda}$.
By \ef{eq:2.14} and Fatou's lemma, one has
$$
\int_{\RN} G_2(u_{\lambda}) \,dx
\le \liminf_{n \to +\infty} \int_{\RN} G_2(u_n^{\lambda}) \,dx.$$
By the weakly lower semi-continuity of $\| \cdot \|_{\CH_0}$ and
from \ef{eq:2.15}, we obtain $I_{\lambda}(u_{\lambda}) \le c_{\lambda}$.
This completes the proof.
\end{proof}

Lemma \ref{lem:2.5} implies that, 
for almost every $\lambda\in J$, $u_{\lambda}$ is a non-trivial solution of \ef{eq:3.1}. 
In order to obtain a non-trivial solution of the original problem \ef{eq:1.1}, 
we next consider a sequence of such $\{ \lambda_n\}$ such that
$\lambda_n \nearrow 1$ as $n \to +\infty$.
Then by Proposition \ref{prop:2.1} and Lemma \ref{lem:2.5}, 
there exists $\{v_n\} \subset \CH_{0,{\rm rad}} \setminus \{ 0\}$ such that
\begin{equation}\label{eq:2.18}
I_{\lambda_n}'(v_n)=0, \ 
I_{\lambda_n}(v_n) \le c_{\lambda_n}.
\end{equation}
Then we claim the following lemma.

\begin{lemma}\label{lem:2.6}
The sequence $\{ v_n\} $ is bounded in $\CH_0$.
\end{lemma}

\begin{proof}
First we observe from $I_{\lambda_n}'(v_n)=0$ that $v_n$ satisfies
$$
-\Delta_p v_n- \Delta_q v_n +g_2(v_n)
-\lambda_n g_1(v_n)=0 \quad \hbox{in} \ \RN,$$
in the weak sense. 
Next we claim that $v_n$ satisfies the following Pohozaev identity:
\begin{equation}\label{eq:2.19}
\frac{N-p}{p} \|\nabla v_n \|_p^p +\frac{N-q}{q} \|\nabla v_n \|_q^q
+N \int_{\RN} G_2(v_n) \,dx
-N\lambda_n \int_{\RN} G_1(v_n) \,dx=0.
\end{equation}
To this aim, we argue as in \cite{PS}.
First adapting the the Moser type iteration as in \cite{HL}, one can show that
$v_n\in C_{\rm loc}^{1,\sigma}(\RN)$ for some $\sigma \in (0,1)$.\footnote{First we notice that the argument used in the proof of Theorem 2 in \cite{HL}
only requires that $v_n \in W^{1,q}_{\rm loc}(\RN)$
because they adopt a cut-off function to obtain desired estimates.
Thus if $q<N$, we can apply Theorems 1-2 in \cite{HL} directly.
If $q>N$, we already have $v_n \in L^{\infty}(\RN)$ so that we can use Theorem 1 in \cite{HL}.
Finally when $q=N$, we have to show that $v_n\in L^{\infty}(\RN)$ first.
But checking the proof of Theorem 2 in \cite{HL} carefully, 
the argument of the proof works even for $q=N$.
Then we can apply Theorem 1 of \cite{HL} to conclude that $v_n \in C_{\rm loc}^{1,\sigma}(\RN)$.}
Next since $v_n \in C^{1,\sigma}_{loc}(\RN)$ and
the function ${\mathcal L}(\xi)=\frac{1}{p} |\xi|^p + \frac{1}{q}|\xi|^q$ 
associated with the differential operator in \ef{eq:1.1} is convex,
we can apply the Pohozaev identity for $C^1$ solutions due to \cite[Lemma 1]{DMS}
by choosing $h(x)=h_k(x)=H(x/k)x \in C_0^1(B_{2k}(0),\R^N)$ for $k \in \N$,
where $H\in C_0^1(\R^N)$ is such that $H(x)=1$ on $|x| \le 1$
and $H(x)=0$ for $|x| \ge 2$.
Letting $k \to +\infty$ and taking into account that 
$|\nabla v_n|^p$, $|\nabla v_n|^q$, $G_1(v_n)$, $G_2(v_n) \in L^1(\RN)$, 
we obtain \ef{eq:2.19} as claimed. 
\\
Now from \ef{eq:2.18}, we have
\begin{align}\label{eq:2.20}
I_{\lambda_n}(v_n)
&=\frac{1}{p} \|\nabla v_n \|_p^p 
+\frac{1}{q} \|\nabla v_n\|_q^q 
+\int_{\RN} G_2(v_n) \,dx 
-\lambda_n \int_{\RN} G_1(v_n) \,dx \le c_{\lambda_n}.
\end{align}
Hence from \ef{eq:2.19}, \ef{eq:2.20} and the monotonicity of
$c_{\lambda}$ with respect to $\lambda$, it follows that
\begin{equation*}
\|\nabla v_n\|_p^p 
+ \|\nabla v_n\|_q^q 
\le N c_{\lambda_n} \le Nc_{\lambda_0}
\end{equation*}
from which we conclude that the assertion holds.
\end{proof}

We can now prove our first main result.

\begin{proof}[Proof of Theorem \ref{thm:1.1}]

By Lemma \ref{lem:2.6}, up to a subsequence, 
we may assume that there exists $v\in \CH_{0, {\rm rad}}$ such that
$v_n \rightharpoonup v$ in $\CH_0$.
Our goal is to show that $v$ is a nontrivial critical point of $I$.
First we prove that $I'(v)=0$. To this aim, we observe from $I_{\lambda_n}'(v_n)=0$ that
$$
I'(v_n)=I_{\lambda_n}'(v_n)+(\lambda_n-1)g_1(v_n)
=(\lambda_n-1)g_1(v_n).$$
Moreover arguing similarly as the proof of \ef{eq:2.17}, one has
$$
\int_{\RN} g_1(v_n) \varphi \,dx
\to \int_{\RN} g_1(v) \varphi \,dx
\ \hbox{for any} \ \varphi \in C_0^{\infty}(\RN).$$
This implies that $(\lambda_n-1)g_1(v_n)=o(1)$ and hence
$\{ v_n \}$ is a Palais-Smale sequence for the functional $I$.
Using the compactness lemma \ref{lem:a.2} again, we can see that $I'(v)=0$.
\\
To conclude the proof, we claim that $v \ne 0$. 
Now from \eqref{eq:2.5}, \eqref{eq:2.6}
and $I_{\lambda_n}'(v_n)=0$, we have
\begin{align} \label{eq:2.21}
\|\nabla v_n\|_p^p + \|\nabla v_n\|_q^q &\le
\|\nabla v_n\|_p^p + \|\nabla v_n\|_q^q
+\int_{\RN} g_2(v_n)v_n \,dx \nonumber \\
&= \lambda_n \int_{\RN} g_1(v_n)v_n \,dx
\ad{ \le C \big( \|v_n\|_{p^*}^{p^*}+\| v_n \|_{q^*}^{q^*} \big)}.
\end{align}
Next we claim that
\begin{equation}\label{eq:2.22}
\liminf_{n \to +\infty} \|  v_n \|_{\CH_0}>0.
\end{equation}
Suppose by contradiction that $v_n \to 0$ in $\CH_0$. Now from \ef{eq:2.4} and \ef{eq:2.21},
one has
$$
\|\nabla v_n\|_p^p+\|\nabla v_n\|_q^q
\le C\left( \ad{ \|\nabla v_n\|_p^{p^*}+\|\nabla v_n\|_p^{q^*}+\| \nabla v_n \|_q^{q^*}} \right).
$$
Since $\| v_n\|_{\CH_0}=\| \nabla v_n\|_p+\| \nabla v_n\|_q \to 0$, 
\ad{$p<p^*<q^*$} and $q<q^*$, we may assume that 
$$
\ad{C(\| \nabla v_n\|_p^{p^*} + \| \nabla v_n\|_p^{q^*})} 
\le \frac{1}{2} \| \nabla v_n\|_p^p,\quad
C\| \nabla v_n \|_q^{q^*} \le \frac{1}{2} \| \nabla v_n \|_q^q,$$
from which we reach a contradiction.

By the compactness lemma \ref{lem:a.2}, one can show that
$$
\int_{\RN} g_1(v_n)v_n \,dx \to \int_{\RN} g_1(v)v \,dx \ \hbox{as} \ n \to +\infty.$$
Then from \eqref{eq:2.21} and \ef{eq:2.22}, we obtain
\begin{equation*}
0< \liminf_{ n \to +\infty} \left(\|\nabla v_n\|_p^p+ \|\nabla v_n\|_q^q \right)
\le \liminf_{n \to +\infty} \lambda_n \int_{\RN} g_1(v_n)v_n \,dx 
= \int_{\RN} g_1(v) v \,dx.
\end{equation*}
This implies that $v \ne 0$ and so we obtain the existence of a nontrivial solution
of \ef{eq:1.1}. 
Applying the the Moser type iteration as in \cite{HL}, 
one has $v \in C^{1,\sigma}_{loc}(\RN)$ for some $\sigma \in (0,1)$.
Then from $g(s) \equiv 0$ for $s\le 0$
and the Harnack inequality due to \cite{T}, it follows that $v>0$ in $\RN$. 
Finally by the radial lemma \ref{lem:a.1}, 
$v \in L^{\infty}(\RN)$ and $v(x) \to 0$ as $|x| \to \infty$.
This completes the proof of Theorem \ref{thm:1.1}.
\end{proof}

We conclude this section by showing the existence of a radial ground state solution of \ef{eq:1.1}.
Let us define by $\CSZ$ the set of the nontrivial radial solutions of \eqref{eq:1.1}, namely
\[
\CSZ=\{u\in \CH_{0,{\rm rad}}\setminus \{0\} \mid I'(u)=0\}.
\]
By Theorem \ref{thm:1.1}, we know that $\CSZ\ne \emptyset$. 
Arguing as in the proof of Theorem \ref{thm:1.1}, we have
\begin{equation}\label{eq:2.23}
\inf_{u\in \CSZ}\|u\|_{\CH_0}>0.
\end{equation}
In a similar argument as in the proof of Lemma \ref{lem:2.6},
any $u\in \CSZ$ satisfies the following Pohozaev identity:
\begin{equation*}
\frac{N-p}{pN} \|\nabla u \|_p^p +\frac{N-q}{qN} \|\nabla u \|_q^q
= \int_{\RN} G(u) \,dx.
\end{equation*}
Thus we infer that
\begin{equation}\label{eq:2.24}
I(u)=\frac{1}{N} \left(\|\nabla u \|_p^p + \|\nabla u \|_q^q \right).
\end{equation}
Combining \eqref{eq:2.23} and \eqref{eq:2.24}, we have that
\begin{equation*}
\sigma=\inf_{u\in \CSZ}I(u)>0.
\end{equation*}

\begin{theorem}\label{thm:2.7}
Assume \eqref{g1}-\ad{\eqref{g4}}. Then \ef{eq:1.1} has a {\em radial ground state solution}, 
namely there exists $\bar u\in \CSZ$ such that 
\[
I(\bar u)=\min_{u\in \CSZ}I(u).
\]
\end{theorem}

\begin{proof}
Let $\{u_n\}\subset \CSZ$ be a minimizing sequence. Since 
\[
I(u_n)=\frac{1}{N} \left(\|\nabla u_n \|_p^p + \|\nabla u_n \|_q^q \right)\to \sigma,
\]
we infer that $\{u_n\}$ is bounded in $\CHZ$. 
Therefore there exists $\bar u\in\CHZ$ such that $u_n \rightharpoonup\bar u$ weakly in $\CHZ$. 
\\
Arguing as in the proof of Theorem \ref{thm:1.1}, 
we have that $\bar u\in \CSZ$ and so we conclude observing that, 
by the weak lower semicontinuity of the norms,
\[
\sigma \le I(\bar u)
=\frac{1}{N} \left(\|\nabla \bar u \|_p^p + \|\nabla \bar u \|_q^q \right)
\le \liminf_{n \to +\infty} \frac{1}{N} \left(\|\nabla u_n \|_p^p + \|\nabla u_n \|_q^q \right)
=\liminf_{n \to +\infty} I(u_n)=\sigma.
\]
This completes the proof.
\end{proof}

\begin{remark}\label{rem:2.8}
In Theorem \ref{thm:2.7}, we could only obtain the existence of a radial ground state solution.
We expect that the existence of ground state solutions 
can be shown without restricting ourselves to the radial class.
For this purpose, we have two possibilities.

One is to characterize the ground state solution as a constraint minimizer of suitable functional.
Then the result on the symmetry of constraint minimizers due to \cite{M}
enables us to conclude that any ground state solution is radially symmetric.
For $0$-th order problem \ef{eq:1.3}, the ground state solution can be characterized as
the following constraint minimizer:
$$
\inf \left\{ \| \nabla u\|_2 \ \Big| \ u\in H^1(\RN), \ \int_{\RN} G(u) \,dx =1 \right\}.
$$
However in order to characterize the ground state solution in this way,
scaling property plays an essential role.
Since scaling argument fails to work in our problem,
we don't know whether the ground state solution of \ef{eq:1.1} can be characterized 
as a constraint minimizer of some suitable functional.

The other possibility is to apply the concentration compactness principle as in \cite{AP2}.
But in order to adopt their argument, 
we also need the characterization of the ground state solution.

It is also worth pointing out that, if the nonlinearity $g(s)$ is 
locally Lipschitz continuous for $s \ge 0$, 
we can apply the symmetry result due to \cite{SZ} for the problem \ef{eq:1.1}, 
to show that any non-negative decaying solution of class $C^1$  is radially symmetric.

Finally if we assume that $g(s)$ is odd as in \cite{BL}, 
we are not able to say that any ground state solution of \ef{eq:1.1} is positive.
This is because generically, the proof of the positivity of ground state solutions
is based on the characterization by constraint minimization,
which is not available for our problem.
\end{remark}

\section{The positive mass case}\label{se:pos}

This section is devoted to the study of \eqref{eq:1.1}, in the positive mass case, 
namely when $g$ satisfies \eqref{g2'} instead of \eqref{g2}.
In this case, we work on the function space $\CH$ which is
given by $\CH= \overline{ C_0^{\infty}(\RN)}^{\ \| \, \cdot \, \|_\CH}$, where
$$
\| u\|_{\CH}:= \|\nabla u\|_p +\|u\|_\ell+\|\nabla u\|_q.
$$
For all $N \ge 3$, it follows that $\CH \hookrightarrow \CH_0$ 
and $\CH \hookrightarrow L^{\ell}(\RN)$.
Thus by Theorem \ref{th:embedding} and since $\ell \in [p,p^*)$, we have
\begin{equation}\label{eq:3.1}
\CH\hookrightarrow L^r(\RN) \ \hbox{ for all }r\in [\ell,q^*].
\end{equation}

For $u\in \CH$, we define the 
functional $I:\CH \to \R$ associated with \ef{eq:1.1} by
$$
I(u)=\frac{1}{p} \|\nabla u\|_p^p 
+\frac{1}{q} \|\nabla u\|_q^q
-\int_{\RN} G(u) \,dx.$$
By hypotheses \eqref{g1}, \eqref{g2'}, \eqref{g3} and from \ef{eq:3.1}, 
we can see that $I$ is well-defined 
and of class $C^1$ on $\CH$ and its critical points are solutions of \eqref{eq:1.1}.

As done in Section \ref{se:0}, we truncate and decompose the nonlinear term $g$. 
Let 
$$
s_0:= \min\{ s\in [\zeta,+\infty) \mid \ g(s)=0 \}$$
and $s_0=+\infty$ if $g(s) \ne 0$ for all $s \ge \zeta$. 
We define $\tilde{g}:\R_+ \to \R$ by
$$
\tilde{g}(s)=\left\{
\begin{array}{ll}
g(s) & \hbox{on} \ [0,s_0],\\
0 & \hbox{on} \ (s_0,+\infty).
\end{array}
\right.$$
Also in this case, by the maximum principle, any positive solutions of \ef{eq:1.1} with $\tilde{g}$ 
satisfy the original problem \ef{eq:1.1} and thus we may replace $g$ by $\tilde{g}$ in \ef{eq:1.1}. 
Hereafter we write $g=\tilde{g}$ for simplicity. 

Next for $s \ge 0$, we set
$$
g_1(s):= (g(s)+m_\ell s ^{\ell-1})_+ \ \ \hbox{and} \ \ 
g_2(s):=g_1(s)-g(s).$$
Then one has $g_1(s)\ge 0$, $g_2(s) \ge 0$ for $s \ge 0$ and 
\begin{equation}\label{eq:3.2}
\lim_{s \to 0} \frac{g_1(s)}{|s|^{\ell-1}}=0,
\end{equation}
\begin{equation}\label{eq:3.3}
\lim_{s \to +\infty} \frac{g_1(s)}{s^{q^*-1}}=0,
\end{equation}
\begin{equation}\label{eq:3.4}
g_2(s) \ge m_\ell s^{\ell-1} \ \hbox{for all} \ s\ge 0.
\end{equation}
From \ef{eq:3.2}-\ef{eq:3.4}, for any $0<\varepsilon<1$, 
there exists $C_{\varepsilon}>0$ such that
\begin{equation}\label{eq:3.5}
g_1(s) \le C_{\varepsilon} |s|^{q^*-1} +\varepsilon g_2(s) \ 
\hbox{for} \ s \ge 0.
\end{equation}
Moreover we put 
$G_i(t)=\int_0^t g_i(s) \,ds$ for $i=1,2$. Then from 
\ef{eq:3.4} and \ef{eq:3.5}, we also have
\begin{equation}\label{eq:3.6}
G_2(s) \ge \frac{m_\ell}{\ell}|s|^\ell \ \hbox{for all} \ s\in \R,
\end{equation}
\begin{equation}\label{eq:3.7}
G_1(s) \le\frac{C_{\varepsilon}}{q^*} |s|^{q^*}+\varepsilon G_2(s)
\ \hbox{for all} \ s\in \R.
\end{equation}

We follow the same strategy as in the previous section and so 
we consider the following auxiliary problem:
\begin{equation}\label{eq:3.8}
-\Delta_p u- \Delta_q u+g_2(u) =\lambda g_1(u) \ \hbox{in} \ \RN
\end{equation}
for $\lambda$ close to $1$. We define the functional 
$I_{\lambda}:\CH \to \R$ by
$$
I_{\lambda}(u)=\frac{1}{p} \|\nabla u\|_p^p
+\frac{1}{q} \|\nabla u\|_q^q 
+\int_{\RN} G_2(u) \,dx
-\lambda \int_{\RN} G_1(u) \,dx.$$
In order to find a non-trivial critical point of $I_{\lambda}$, 
we apply the Monotonicity trick (see Proposition \ref{prop:2.1}), 
with $X=\CH_{{\rm rad}}$, where
\[
\CH_{{\rm rad}}=\{u\in \CH\mid u \hbox{ is radially symmetric }\},
\]  
and 
\begin{align*}
A(u)&=\frac{1}{p} \|\nabla u\|_p^p 
+\frac{1}{q} \|\nabla u\|_q^q 
+\int_{\RN} G_2(u) \,dx,
\\
B(u) &=\int_{\RN} G_1(u) \,dx.
\end{align*}

Now arguing as in Lemma \ref{lem:2.2}, we infer the following.
\begin{lemma}\label{lem:3.1}
There exists $\lambda_0 \in (0,1)$ such that 
the set $\Gamma_\l$ defined in \ef{eq:2.11} is non-empty 
for every $\lambda \in J=[\lambda_0,1]$.
\end{lemma}

Next we establish the following lemma by modifying the proof of Lemma \ref{lem:2.3}.
\begin{lemma}\label{lem:3.2}
For all $\lambda \in J=[\lambda_0,1]$, the condition \ef{eq:2.12} holds.
\end{lemma}

\begin{proof}
For any $u\in \CH_{{\rm rad}}$ and $\lambda \in J$, 
we have from \eqref{eq:3.6} and \eqref{eq:3.7}, and later by \eqref{eq:2.4}, that
\begin{align*}
I_{\lambda}(u) &\ge \frac{1}{p}\|\nabla u\|_p^p 
+\frac{1}{q} \|\nabla u\|_q^q 
+\frac{m_\ell(1-\varepsilon)}{\ell}\|u\|_\ell^\ell 
-\frac{C_\varepsilon}{q^*} \|u\|_{q^*}^{q^*} \\
&\ge \frac{1}{p} \| \nabla u\|_p^p
\Big(1-C\| \nabla u\|_p^{p^*-p}-C \| \nabla u\|_p^{q^*-p}\Big)
+\frac{1}{q} \| \nabla u\|_q^q \Big( 1-C\| \nabla u\|_q^{q^*-q}\Big)
+\frac{m_\ell(1-\varepsilon)}{\ell}\|u\|_\ell^\ell.
\end{align*}
Let $u\in \CH_{{\rm rad}}$ such that 
$\|u\|_{\CH}=\| \nabla u\|_p +\| u\|_{\ell}+\| \nabla u\|_q=\rho < 1$. 
Since \ad{$q^* > p^*>p$} and $q^*>q$, if $\rho>0$ is sufficiently small, 
it follows that
$$
\ad{1-C\| \nabla u\|_p^{p^*-p}-C \| \nabla u\|_p^{q^*-p}\ge \frac{1}{2}}, \quad
1-C\| \nabla u\|_q^{q^*-q} \ge \frac{1}{2}.$$
Then from $p \le \ell$ and $p<q$, we get
$$
I_{\lambda}(u) \ge C \| u\|_{\CH}^{\bar{\ell}},$$
where $\bar{\ell}=\max \{ \ell,q\}$.
Therefore there exists $\delta>0$ such that $I_\l(u)\ge \delta$
for all $u\in \CH_{{\rm rad}}$ with $\|u\|_{\CH}\le\rho$.
The conclusion follows as in Lemma \ref{lem:2.3}.
\end{proof}

By Lemmas \ref{lem:3.1} and \ref{lem:3.2}, we can apply 
Proposition \ref{prop:2.1} to obtain a bounded Palais-Smale
sequence $\{ u_n^{\lambda}\} \subset \CH_{{\rm rad}}$ of $I_{\lambda}$, for almost every $\lambda\in J$, that is,
$$
I_{\lambda}(u_n^{\lambda}) \to c_{\lambda}, \ 
I_{\lambda}'(u_n^{\lambda}) \to 0 \ \hbox{and} \ 
\{ u_n^{\lambda}\} \ \hbox{is bounded in }\CH$$
Hence, passing to a subsequence, there exists $u_{\lambda} \in \CH_{{\rm rad}}$ such that
\begin{align}
&u_n^{\lambda} \rightharpoonup u_{\lambda}\ \hbox{ in }\ \CH,
\ \hbox{ as } n \to +\infty, \nonumber \\ 
&u_n^{\lambda}(x) \to u_{\lambda}(x)  \ 
\hbox{ a.e. } \ x\in \RN, \ \hbox{ as} \ n \to +\infty. \nonumber
\end{align}

\begin{lemma}\label{lem:3.3}
The weak limit $u_{\lambda}$ satisfies
$$
u_{\lambda} \ne 0, \quad I_{\lambda}'(u_{\lambda})=0 \quad \hbox{and} \quad
I_{\lambda}(u_{\lambda}) \le c_{\lambda}.$$
\end{lemma}

\begin{proof}
The proof is almost same as that of Lemma \ref{lem:2.5}.
The only difference is the choice of $Q(s)$ to apply the Strauss's compactness lemma.
Indeed in the positive mass case, putting $Q(s)=|s|^{\ell}+|s|^{q^*}$, 
one has from \ef{eq:3.2} and \ef{eq:3.3} that
$\frac{G_1(s)}{Q(s)} \to 0$ as $s \to 0$ and $s \to \infty$.
The rest of the proof can be done in a similar way as Lemma \ref{lem:2.5}.
\end{proof}

Lemma \ref{lem:3.3} implies that, for almost every $\lambda\in J$, $u_{\lambda}$ is a non-trivial 
solution of \ef{eq:3.8}. 
In order to obtain a non-trivial solution of the
original problem \ef{eq:1.1}, we next consider a sequence
$\{ \lambda_n\}$ such that
$\lambda_n \nearrow 1$ as $n \to +\infty$.
Then by Proposition \ref{prop:2.1} and Lemma \ref{lem:3.3}, 
there exists $\{v_n\} \subset \CH_{{\rm rad}} \setminus \{ 0\}$ such that
\begin{equation}\label{eq:3.9}
I_{\lambda_n}'(v_n)=0, \ 
I_{\lambda_n}(v_n) \le c_{\lambda_n}.
\end{equation}
Then we claim the following lemma.

\begin{lemma}\label{lem:3.4}
The sequence $\{ v_n\} $ is bounded in $\CH$.
\end{lemma}

\begin{proof}
The conclusion follows by the same argument as in the proof of Lemma \ref{lem:2.6}.
But in the positive mass case, we further need a bound for the $L^\ell$-norm of $\{ v_n \}$.
\\
Now since $I_{\lambda_n}'(v_n)=0$, it follows that $v_n$ satisfies
$$
-\Delta_p v_n- \Delta_q v_n +g_2(v_n)
-\lambda_n g_1(v_n)=0 \quad \hbox{in} \ \RN,$$
in the weak sense. 
Then as in the zero mass case, one can show that
the following Pohozaev identity holds:
\begin{equation}\label{eq:3.10}
\frac{N-p}{p} \|\nabla v_n \|_p^p +\frac{N-q}{q} \|\nabla v_n \|_q^q
+N \int_{\RN} G_2(v_n) \,dx
-N\lambda_n \int_{\RN} G_1(v_n) \,dx=0.
\end{equation}
Moreover from \ef{eq:3.9}, we also have
\begin{align}\label{eq:3.11}
I_{\lambda_n}(v_n)
&= \frac{1}{p} \|\nabla v_n \|_p^p 
+\frac{1}{q} \|\nabla v_n\|_q^q 
+\int_{\RN} G_2(v_n) \,dx 
-\lambda_n \int_{\RN} G_1(v_n) \,dx \le c_{\lambda_n}, \\
\label{eq:3.12}
I_{\lambda_n}'(v_n)[v_n]
&= \|\nabla v_n\|_p^p 
+ \|\nabla v_n\|_q^q
+\int_{\RN} g_2(v_n)v_n \,dx
-\lambda_n \int_{\RN} g_1(v_n)v_n \,dx =0.
\end{align}
From \ef{eq:3.10}, \ef{eq:3.11} and the monotonicity of
$c_{\lambda}$ with respect to $\lambda$, it follows that
\begin{equation*}
\|\nabla v_n\|_p^p 
+ \|\nabla v_n\|_q^q 
\le N c_{\lambda_n} \le Nc_{\lambda_0},
\end{equation*}
and hence
\begin{equation}\label{eq:3.13}
\|v_n\|_{\CH_0} \le C.
\end{equation}
To conclude, we have to show that $\{v_n\}$ is bounded in $L^{\ell}(\RN)$. 
By \ef{eq:3.4}, \ef{eq:3.5} and \ef{eq:3.12}, one has
\begin{align*}
0&=\|\nabla v_n\|_p^p
+ \| \nabla v_n\|_q^q 
+\int_{\RN} g_2(v_n)v_n \,dx
-\lambda_n \int_{\RN} g_1(v_n) v_n \,dx \\
&\ge (1-\varepsilon)m_\ell \| v_n \|^\ell_\ell
-C_{\varepsilon} \| v_n\|^{q^*}_{q^*}.
\end{align*}
By Theorem \ref{th:embedding} and \ef{eq:3.13}, we get
\begin{equation*}
\|v_n\|^\ell_\ell
\le C\| v_n\|^{q^*}_{q^*}
\le C \| v_n\|_{\CH_0}^{q^*}
\le C.
\end{equation*}
This, together with \eqref{eq:3.13}, completes the proof.
\end{proof}

\begin{proof}[Proof of Theorem \ref{thm:1.2}]
By Lemma \ref{lem:3.4}, up to a subsequence, 
we may assume that there exists $v\in \CH_{{\rm rad}}$ such that $v_n \rightharpoonup v$ in $\CH$. 
Arguing as in the proof of Theorem \ref{thm:1.1}, 
we can see that $v$ is a nontrivial critical point of $I$. 
Moreover by the regularity theory and the Harnack inequality, it follows that $v>0$ in $\RN$.
\end{proof}

To introduce the existence of a ground state solution, 
let us define $\CS$ the set of the nontrivial radial solutions of \eqref{eq:1.1}, namely
\[
\CS=\{u\in \CH_{{\rm rad}}\setminus \{0\} \mid I'(u)=0\}.
\]
Arguing as Theorem \ref{thm:2.7}, one can show the following result.


\begin{theorem}\label{thm:3.5}
Assume \eqref{g1}, \eqref{g2'}, \eqref{g3} and \eqref{g4}. 
Then \ef{eq:1.1} has a {\em radial ground state solution}, 
namely there exists $\bar u\in \CS$ such that 
\[
I(\bar u)=\min_{u\in \CS}I(u).
\]
\end{theorem}



\section{$k$-th order approximated problem}\label{se:appr}

In this section, we consider the approximation, at $k$-th order, of the Born-Infeld  equation \eqref{eq:1.3}, namely we deal with the following problem:
\begin{equation} \label{eq:4.1}
\begin{cases}
-\Delta u-\Delta_4u 
\cdots - \Delta_{2k} u =h(u) & \hbox{in} \ \RN,
\\
u(x)\to 0 & \hbox{as }|x|\to +\infty,
\end{cases}
\end{equation}
where $N \ge 3$, $k \in \N$ with $k \ge 2$.
Here we normalized the coefficients $\frac{(2k-3)!!}{(k-1)!}\beta^{k-1}$ in \ef{eq:1.4}
because they are not essential for the existence of solutions.
Moreover since we are interested in higher order approximation, we assume that
\begin{equation} \label{eq:4.2}
k \ge \max \left\{ \frac{N}{2}, \frac{N}{N-2} \right\}.
\end{equation}
We impose the following assumptions on $h$.
\begin{enumerate} [label=(h\arabic*),ref=h\arabic*]
\item \label{h1} $h\in C(\R_+,\R)$ and $h(s) \equiv 0$ for $s\le 0$.
\item \label{h2} Either (i) or (ii) is fulfilled:
\begin{enumerate}
\item[(i)] for all $\ell\in [2, \frac{2N}{N-2}]$, it holds
$$
-\infty
\le \limsup_{s \to 0^+} \frac{h(s)}{s^{\ell-1}} \le 0;$$
\item[(ii)] there exist $\ell\in [2, \frac{2N}{N-2})$ and $m_{\ell}>0$ such that 
$$
-\infty<\liminf_{s\to 0^+} \frac{h(s)}{s^{\ell-1}} \le
\limsup_{s \to 0^+} \frac{h(s)}{s^{\ell-1}} =-m_\ell.$$
\end{enumerate}
\item \label{h3} There exists $\ell^*>2k\ge \frac{2N}{N-2}$ such that
\begin{equation*} 
\displaystyle -\infty \le \limsup_{s \to +\infty}
\frac{h(s)}{s^{{\ell}^*-1}} \le 0.
\end{equation*}
\item \label{h4} There exists $\zeta>0$ such that $H(\zeta)=\int_0^{\zeta} h(s) \,ds >0$.
\end{enumerate}
In this setting, we have the following result.

\begin{theorem} \label{thm:4.1}
Assume \ef{eq:4.2}, \eqref{h1}-\ef{h4}.
Then problem \ef{eq:4.1} has a solution which is positive and radially symmetric and belongs to  
$C_{\rm loc}^{1,\sigma}(\RN)\cap L^{\infty}(\RN)$ class,
for some $\sigma \in (0,1)$. 
Furthermore, there exists a radial ground state solution of \ef{eq:4.1}.
\end{theorem}

As a special case, let us study the problem:
\begin{equation} \label{eq:4.3}
\begin{cases}
\displaystyle -\Delta u - \beta \Delta_4 u -\frac{3}{2}\beta^2 \Delta_6 u
\cdots -\frac{(2k-3)!!}{(k-1)!}\beta^{k-1} \Delta_{2k} u = |u|^{\alpha-1}u
& \hbox{in} \ \RN,
\\[4mm]
u(x)\to 0 & \hbox{as }|x|\to +\infty,
\end{cases}
\end{equation}
for $N \ge 3$, $\alpha> \frac{2N}{N-2}$ and $\beta>0$.
Under the assumption \ef{eq:4.2}, we choose $\ell^*> \max \{ \alpha, 2k \}$ arbitrarily and, by Theorem \ref{thm:4.1}, we obtain the following result.

\begin{corollary} \label{cor:4.2}
Assume \ef{eq:4.2} and let $\alpha>\frac{2N}{N-2}$ and $\beta>0$ be arbitrarily given.
Then the problem \ef{eq:4.3} has a positive radial solution 
as well as a radial ground state solution.
\end{corollary}
We expect that under some smallness condition on $\beta$, a positive solution $u_k$ of \ef{eq:4.3}
converges to a positive solution of
\begin{equation} \label{eq:4.4}
\begin{cases}
\displaystyle -\Div \left( \frac{\nabla u}{\sqrt{1-2\beta|\nabla u|^2}} \right) =|u|^{\alpha-1}u
& \mbox{in} \ \RN,
\\[4mm]
u(x)\to 0 & \hbox{as }|x|\to +\infty,
\end{cases}
\end{equation}
as $k \to +\infty$ in a certain sense.
But we postpone this question to a future work.

We also note that the problem \ef{eq:4.3} has no non-trivial $C^1$ solution 
if $1<\alpha \le \frac{2N}{N-2}$.
Indeed for any non-trivial $C^1$ solution of \ef{eq:4.3}, 
by adapting the argument in \cite{DMS},
one can prove the following two identities hold:
\begin{align*}
\| \nabla u\|_2^2 + \beta \| \nabla u\|_4^4 \cdots 
+ \frac{(2k-3)!!}{(k-1)!}\beta^{k-1} \| \nabla u\|_{2k}^{2k}&
=\| u\|_{\alpha}^{\alpha} \quad (\hbox{Nehari}) \\
\frac{N-2}{2} \| \nabla u\|_2^2 + \frac{N-4}{4} \beta \| \nabla u\|_4^4
\cdots + \frac{N-2k}{2k} \frac{(2k-3)!!}{(k-1)!}\beta^{k-1} \| \nabla u\|_{2k}^{2k} &
=\frac{N}{\alpha} \| u\|_{\alpha}^{\alpha} 
\quad (\hbox{Pohozaev}).
\end{align*}
Substituting the first equation for the second one, we obtain
\begin{multline} \label{eq:4.5}
\left( \frac{N-2}{2}-\frac{N}{\alpha} \right) \| \nabla u\|_2^2
+\left( \frac{N-4}{4}-\frac{N}{\alpha} \right) \beta \| \nabla u\|_4^4 \\
+\cdots +\left( \frac{N-2k}{2k}-\frac{N}{\alpha} \right) 
\frac{(2k-3)!!}{(k-1)!}\beta^{k-1}\| \nabla u\|_{2k}^{2k}=0.
\end{multline}
If $1<\alpha \le \frac{2N}{N-2}$, it follows that $\frac{N}{\alpha} \ge \frac{N-2}{2}$ and hence
$$
\frac{N-2j}{2j}-\frac{N}{\alpha} \le \frac{N-2j}{2j}-\frac{N-2}{2} 
=-\frac{(j-1)N}{2j} \le 0 \quad \hbox{for} \ j \ge 1.$$
This implies that all terms in the left hand side of \ef{eq:4.5} are non-negative,
yielding that $\nabla u \equiv 0$ and hence $u \equiv 0$.
We note that the non-existence of positive radial solutions of \ef{eq:4.4}
for the case $1<\alpha  \le \frac{2N}{N-2}$ has been obtained in \cite{A2}
by the ODE technique.

\medskip
The proof of Theorem \ref{thm:4.1} is almost the same as those of Theorems \ref{thm:1.1}-\ref{thm:1.2},
\ref{thm:2.7} and \ref{thm:3.5}.
Here we consider the zero mass case (h2-i) and only give a sketch of the proof.

First we set a function space $\mathcal{H}^{2,2k}_{0}$ defined by
$\mathcal{H}^{2,2k}_{0}= 
\overline{ C_0^{\infty}(\RN)}^{\ \| \, \cdot \, \|_{\mathcal{H}^{2,2k}_{0}}}$, where
$$
\| u\|_{\mathcal{H}^{2,2k}_{0}} := \| \nabla u\|_2 + \| \nabla u\|_{2k}.$$
Since $2k \ge N$, it follows by Theorem \ref{th:embedding} that 
$\mathcal{H}^{2,2k}_{0} \hookrightarrow L^r(\RN)$ and
\begin{equation} \label{eq:4.6}
\| u\|_r \le C( \| \nabla u\|_2 + \| \nabla u\|_{2k})
\ \hbox{for any} \ u\in \mathcal{H}^{2,2k}_{0} \ \hbox{and} \ 
r\in \left[\frac{2N}{N-2},+\infty\right).
\end{equation}
We define the functional $I_k:\mathcal{H}^{2,2k}_{0} \to \mathbb{R}$ by
$$
I_k(u)=\frac{1}{2} \| \nabla u\|_2^2+\frac{1}{4} \| \nabla u\|_4^4
\cdots + \frac{1}{2k} \| \nabla u\|_{2k}^{2k} - \int_{\RN} H(u) \,dx,$$
which is well-defined and $C^1$ by \eqref{h1}-\eqref{h3}.
Moreover we truncate and decompose $h(s)$ as in Section \ref{se:vs}.
We apply the Monotonicity trick to $X=\mathcal{H}^{2,2k}_{0,{\rm rad}}$, where
\[
\mathcal{H}^{2,2k}_{0,{\rm rad}}=\{u\in \mathcal{H}^{2,2k}_0\mid u \hbox{ is radially symmetric }\},
\] and consider the modified functional $I_{k,\lambda}$
which is given by
$$
I_{k,\lambda}(u)=\frac{1}{2} \| \nabla u\|_2^2+\frac{1}{4} \| \nabla u\|_4^4
\cdots + \frac{1}{2k} \| \nabla u\|_{2k}^{2k} 
+ \int_{\RN} H_2(u) \,dx - \lambda \int_{\RN} H_1(u) \,dx$$
for $\lambda \in (0,1]$. 

The arguments from now on are similar to those of the previous sections and we omit the details.

\appendix
\section{}
In this appendix, we collect some well known lemmas which we used in this paper.

\begin{lemma}[Radial Lemma, \cite{BL,SWW}]
\label{lem:a.1}

Suppose $1<p<N$. Then there exists $C=C(N,p)>0$ such that
for any $u \in D^{1,p}_{\rm rad}(\RN)$,
$$
|u(x)| \le C |x|^{-\frac{N-p}{p}} \| \nabla u\|_p.$$
\end{lemma}

%

Next we recall a variant of the Strauss' compactness lemma due to \cite{AP}.
(See also \cite[Theorem A.1]{BL}, \cite{Str}.) 
It will be a fundamental tool in our arguments.
\begin{lemma}\label{lem:a.2}
Let $P$ and $Q:\R\to\R$ be two continuous functions satisfying
\begin{equation*}
\lim_{s\to\infty}\frac{P(s)}{Q(s)}=0,
\end{equation*}
$\{v_n\},$ $v$ and $z$ be measurable functions from $\RN$ to $\R$, with $z$ bounded,
such that
\begin{align*}
&\sup_{n \in \mathbb{N}} \int_\RN | Q(v_n(x))z|\,dx <+\infty,
\\ 
&P(v_n(x))\to v(x) \:\hbox{a.e. in }\RN. 
\end{align*}
Then $\|(P(v_n)-v)z\|_{L^1(B)}\to 0$, for any bounded Borel set
$B$.

Moreover, if we have also
\begin{align*}
\lim_{s\to 0}\frac{P(s)}{Q(s)} &=0,\\ 
\lim_{x\to\infty}\sup_{n \in \mathbb{N}} |v_n(x)| &= 0, 
\end{align*}
then $\|(P(v_n)-v)z\|_{L^1(\RN)}\to 0.$
\end{lemma}

\subsection*{Acknowledgment}
The authors would like to thank the anonymous referee for useful suggestions.
The first author is partially supported by  a grant of the group GNAMPA of INdAM. 
The second author is supported by JSPS Grant-in-Aid for Scientific Research (C) (No. 15K04970).

\end{document}